\documentclass[10pt, letterpaper]{article}

\usepackage[pdftex]{graphicx,color}  
\usepackage{amsfonts, amssymb, amscd, amsmath, amsthm, fullpage}
\usepackage{enumerate}

\newtheorem*{bianchi}{Theorem \ref{bianchi}}
\newtheorem{theorem}{Theorem}[section]
\newtheorem{cor}[theorem]{Corollary}
\newtheorem{lemma}[theorem]{Lemma}

\newcommand{\N}{\mathbb{N}}

\newcommand{\R}{\mathbb{R}}
\newcommand{\Q}{\mathbb{Q}}
\newcommand{\Z}{\mathbb{Z}}
\newcommand{\C}{\mathbb{C}}
\newcommand{\Hy}{\mathbb{H}}
\newcommand{\Hyp}{\mathbb{H}^2}
\newcommand{\Hys}{\mathbb{H}^3}

\newcommand{\pslz}{\mathrm{PSL}_2(\Z)}
\newcommand{\pslr}{\mathrm{PSL}_2(\R)}
\newcommand{\pslc}{\mathrm{PSL}_2(\C)}
\newcommand{\pslod}{\mathrm{PSL}_2(\mathcal{O}_d)}

\newcommand{\tr}{\mathrm{tr} \, }
\newcommand{\gamo}{\Gamma_0}

\title{Equivalent trace sets for arithmetic Fuchsian groups}

\author{Grant S. Lakeland}

\begin{document}

\maketitle

\begin{abstract}We show that the modular group has an infinite family of finite index subgroups, each of which has the same trace set as the modular group itself. Various congruence subgroups of the modular group, and the Bianchi groups, are also shown to have this property. In the case of the modular group, we construct examples of such finite index subgroups. \end{abstract}

\section{Introduction}

For a Riemannian manifold $M$, the \emph{eigenvalue spectrum} is the set of eigenvalues of the Laplacian operator, and the \emph{length spectrum} is the set of lengths of closed geodesics, both counted with multiplicity. The two spectra are closely related, and together determine much about the manifold: though constructions such as Sunada's \cite{Sunada} show that there exist isospectral, non-isometric manifolds, it is known that manifolds for which these spectra are equal must share certain geometric and topological properties; for example, if the manifolds are hyperbolic, they must have the same volume \cite{Pesce}. It is also possible to define the \emph{eigenvalue set} $E(M)$ and the \emph{length set} $L(M)$ to be the respective spectra with multiplicities discarded. It is known that these form invariants which are considerably coarser; for example, Leininger, McReynolds, Neumann, and Reid \cite{LMNR} proved that if $M$ is a compact hyperbolic manifold, then there exist sequences of pairs of covers $\{ M_i, N_i \}$ such that for all $i$, $E(M_i)=E(N_i)$ and $L(M_i)=L(N_i)$, but the ratio $\frac{\mathrm{Vol}(M_i)}{\mathrm{Vol}(N_i)}$ diverges to $\infty$. 

When the manifold in question is a hyperbolic surface, so that the fundamental group $\pi_1(M)$ is a discrete subgroup of $\pslr$, it is well-known that the length $\ell$ of a closed geodesic determines, and is determined by, (the absolute value of) the trace $\tr$ of the corresponding hyperbolic isometry via the equation
\[ | \tr | = 2 \cosh{\frac{\ell}{2}}. \]
There is a similar correspondence when $M$ is a hyperbolic $3$--manifold, so that $\pi_1(M) < \pslc$ acts on the hyperbolic $3$--space $\Hy^3$. Here a complex trace corresponds to a \emph{complex length} in $M$; the corresponding action on $\Hy^3$ involves a combination of translation along and rotation around the axis of the isometry. In these cases, the length set (resp. complex length set) of the manifold $M = \Hy^n/\Gamma$ is in direct correspondence with the trace set of the Fuchsian (resp. Kleinian) group $\Gamma$. 

It is therefore a consequence of the aforementioned result of Leininger, McReynolds, Neumann and Reid that there exist pairs of Fuchsian and Kleinian groups of different covolumes but with equal trace sets. One may then ask the question of, given a prescribed trace set, how many (if any) groups possess precisely that trace set. In particular, are there any Fuchsian or Kleinian groups which are uniquely determined by their trace set? In this direction, Schmutz \cite{Schmutz} showed that there are infinitely many Fuchsian groups with the same trace set as certain congruence subgroups of the modular group $\pslz$. In a similar direction, the goal of this note is to determine to what extent $\pslz$, and certain subgroups thereof, are determined by their trace sets. It is a consequence of Takeuchi's characterization of arithmetic Fuchsian groups \cite{Takeuchi} that any (cofinite) Fuchsian group with trace set precisely the rational integers $\Z$ must be arithmetic, and in fact (conjugate to) a subgroup of $\pslz$ (see also Geninska and Leuzinger \cite{GenLeu}). We show that there are infinitely many such finite index subgroups by constructing a finitely generated, infinite index subgroup $H$ with the same trace set as the modular group, and appealing to the fact that the modular group has the property of being \emph{subgroup separable}, also called locally extended residually finite (LERF), which implies the existence of one (and hence an infinite descending chain of) finite index subgroup(s) containing $H$.

More generally, in Section \ref{proofs}, we show the following.

\begin{theorem}\label{mainthm}Let $\Gamma < \pslr$ be a cofinite Fuchsian group with trace set $\tr(\Gamma)$. Let $G_1, \ldots, G_m$ be a finite collection of finitely generated, infinite index subgroups of $\Gamma$ such that 
\[ \bigcup_{i=1}^m \tr(G_i) = \tr(\Gamma). \] 
Then, for each $1 \leq i \leq m$, there exists $\alpha_i \in \Gamma$ such that the subgroup $H$ of $\Gamma$ generated by the conjugates $\alpha_i G_i \alpha_i^{-1}$ is a subgroup of $\Gamma$ of infinite index, with $\tr(H) = \tr(\Gamma)$. \end{theorem}

Theorem \ref{mainthm} is then applied to $\pslz$, and to families of congruence subgroups thereof.

\begin{cor}\label{pslzcor}Let $\Gamma = \pslz$, or a congruence subgroup $\gamo(n)$ or $\Gamma(n)$ for some $n \in \N$. Then there exists a finitely generated, infinite index subgroup $H_\Gamma < \Gamma$ with the same trace set as $\Gamma$. Hence, there exist infinitely many finite index subgroups of $\Gamma$ with this trace set.\end{cor}

In Section \ref{examples}, we give two constructions, each of which gives rise to an infinite family of examples of finite index subgroups of $\pslz$ with the same trace set as $\pslz$. These examples show that such subgroups are not easily characterized: the first construction implies that all such subgroups cannot be arranged as finitely many descending nested chains of subgroups; the second implies that the multiplicities of a finite set of traces may be chosen to be arbitrarily high.

The most natural lattices in $\pslc$ which serve as analogues of the modular group are the Bianchi groups $\pslod$, where $d >0$ is a square-free integer, and $\mathcal{O}_d$ is the ring of integers in the imaginary quadratic number field $\Q(\sqrt{-d})$. In Section \ref{bianchisection}, we show a similar result for these groups.

\begin{theorem}\label{bianchi}Given any Bianchi group $\pslod$, there are infinitely many finite index subgroups $\Gamma < \pslod$ with the same complex trace set as $\pslod$.\end{theorem}

\noindent {\bf Acknowledgments.} I wish to thank Chris Leininger and Alan Reid for helpful conversations, Mark Bell for computational assistance, and the referee for valuable comments.

\section{Preliminaries}\label{prelims}

We refer to Beardon \cite{Beardon} for more details of the contents of this section. We consider the upper half-plane and upper half-space models for hyperbolic $2$- and $3$-space $\Hyp$ and $\Hys$ respectively. The group of conformal, orientation-preserving isometries (or linear fractional transformations) of $\mathbb{H}^2$ (resp. $\Hys$) can be identified with $\mathrm{PSL}_2(\mathbb{R})$ (resp. $\pslc$) via the correspondence
\[ \begin{pmatrix} a & b \\ c & d\end{pmatrix}  \\ \ \ \  \longleftrightarrow \ \ \ z \longmapsto \frac{az+b}{cz+d}. \]
Given an element $\gamma \in \pslc$, the trace $\tr \gamma$ is not well-defined, but is well-defined up to sign. Given a discrete group $\Gamma < \pslc$, we define the \emph{trace set} of $\Gamma$ to be
\[ \tr (\Gamma) = \left\lbrace \pm \tr \tilde{\gamma} \mid \tilde{\gamma} \in \mathrm{SL}_2(\mathbb{C}) \ \mbox{is a lift of} \ \gamma \in \Gamma \setminus \{ 1 \} \right\rbrace/(x \sim -x). \] 
For every $n \in \Z$, the modular group $\pslz$ has an element with trace $n$; taking the above into account, throughout this note we will say that $\pslz$ has trace set $\N_0 = \N \cup \{ 0 \}$.

The \emph{Farey tessellation} of $\Hyp$ tiles the hyperbolic plane by ideal triangles. There is a vertex for each point of $\Q \cup \{ \infty \}$, and an edge between two vertices $p/q$ and $r/s$ whenever the determinant $ps-qr = \pm1$ (the fractions are assumed to be in lowest terms, with $\infty = 1/0$). The group $\pslz$ acts transitively on the tiles of the Farey tessellation, and also acts transitively on the edges. For each tile, there is a unique order three rotation belonging to $\pslz$ which permutes the vertices and edges of that tile; for each edge, there exists a unique involution belonging to $\pslz$ which fixes a point on that edge.

If $|\tr \gamma| <2$, then $\gamma$ is elliptic and fixes a point of $\Hyp$, or fixes an axis of $\Hys$ pointwise. If $|\tr \gamma| =2$, then $\gamma$ is parabolic and fixes exactly one point on the boundary circle or sphere. If $|\tr \gamma| >2$, then $\gamma$ is hyperbolic, fixes two points on the boundary circle or sphere, and acts as a translation along the geodesic between these two fixed points. If $\tr \gamma \notin \R$, then $\gamma$ is loxodromic, fixes two points on the boundary, and acts as both a translation along, and a rotation around, the axis between the two fixed points.

We can study the action of an element $\gamma \in \pslc$ which does not fix $\infty$ on the upper half-plane model for $\Hyp$ or $\Hys$ by considering \emph{isometric circles} or \emph{isometric spheres}. Given 
\[ \gamma = \begin{pmatrix}a & b \\ c & d \end{pmatrix} \in \pslc , \]
where $c \neq 0$ (since we assume $\gamma$ does not fix $\infty$) the isometric sphere $S_\gamma$ of $\gamma$ has center $\frac{-d}{c}$ and radius $\frac{1}{|c|}$; furthermore, the isometric sphere $S_{\gamma^{-1}}$ has center $\frac{a}{c}$ and the same radius $\frac{1}{|c|}$. The action of $\gamma$ on $S_\gamma$ is by a Euclidean isometry, and $\gamma$ sends the exterior $E_\gamma$ (resp. interior $I_\gamma$) of $S_\gamma$ to the interior $I_{\gamma^{-1}}$ (resp. exterior $E_{\gamma^{-1}}$) of $S_{\gamma^{-1}}$. In particular, we note that when $\tr \gamma = 0$, corresponding to a rotation of order $2$, the isometric spheres $S_\gamma$ and $S_{\gamma^{-1}}$ coincide, and $\gamma$ then acts by exchanging the interior and exterior of $S_\gamma$. In the following, when we refer to a closure $\overline{P}$ of a set $P \subset \Hyp$, we mean the closure taken in $\Hyp \cup \R \cup \infty$.

When $\Gamma$ contains a parabolic element fixing $\infty$, the set of isometric circles of elements of $\Gamma$ is invariant under this subgroup. In this case, one may construct a fundamental domain, called a \emph{Ford domain}, by taking the set $E$ of points exterior to all isometric circles, and intersecting it with a fundamental region for the stabilizer of $\infty$. When $\Gamma < \pslc$ is generated by (at most two) parabolics fixing $\infty$ and a single element $\gamma$ which does not fix $\infty$, then we have two cases of how a Ford domain may be constructed. If the isometric spheres of $\gamma$ do not intersect, or if $\tr \gamma \in \R$, then they (and their translates) suffice to form the boundary of a Ford domain. If not, then the isometric spheres of powers of $\gamma$ may not be covered by those of $\gamma$, and so appear in the boundary of a Ford domain. In this case, we will use the properties that any isometric circle of $\gamma^{\pm n}$ contains one of the fixed points of $\gamma$, and that if $\gamma$ belongs to a Bianchi group, then the isometric spheres have radius bounded above by $1$.

Given two non-cofinite Fuchsian or Kleinian groups $G_1, G_2 < G$, the Klein--Maskit combination theorem gives a way of ensuring that the group $\left\langle G_1, G_2 \right\rangle$ generated by these two subgroups inside of $G$ is also non-cofinite. Precisely, it states (see Maskit \cite{Maskit}, p. 139):

\begin{theorem}[Klein--Maskit Combination Theorem]\label{combthm} Suppose $G_1, G_2 < G$ have fundamental domains $D_1$ and $D_2$ respectively, and that $D_1 \cup D_2 = \Hy^n$ ($n=2,3$) and $D_1 \cap D_2 \neq \emptyset$. Then $\left\langle G_1, G_2 \right\rangle = G_1 \ast G_2$, and $D = D_1 \cap D_2$ is a fundamental domain for $G_1 \ast G_2$. \end{theorem}


A group $G$ is called \emph{residually finite} if for any non-trivial element $g \in G$, there is a finite index subgroup $K<G$ such that $g \notin K$. The group $G$ has the stronger property of being \emph{subgroup separable} (or LERF) if for any finitely generated subgroup $H<G$ and any $g \in G \setminus H$, there exists a finite index subgroup $K<G$ such that $H \subset K$ and $g \notin K$. Equivalently, $G$ is LERF if every finitely generated subgroup $H<G$ is the intersection of finite index subgroups of $G$; we will appeal to this alternative formulation. The fact that $\pslz$ is LERF follows from the fact that it contains a free group of finite index, and Hall's result \cite{Hall} that free groups are LERF; the fact that the Bianchi groups are LERF follows from work of Agol, Long and Reid \cite{ALR}, Agol \cite{Agol}, Calegari and Gabai \cite{CG}, and Canary \cite{Canary}.

Given a natural number $n$, the \emph{principal congruence subgroup} $\Gamma(n) < \pslz$ consists of those matrices which are congruent to the identity modulo $n$; that is
\[ \Gamma(n) = \mathrm{P} \left\{ \begin{pmatrix} a & b \\ c & d \end{pmatrix} \in \mathrm{SL}_2(\mathbb{Z}) \ \middle| \ a \equiv d \equiv 1 \ \mbox{and} \ b \equiv c \equiv 0 \ \mbox{mod} \ n \right\}. \]
All principal congruence subgroups are finite index and normal in $\pslz$, since they are the kernels of the natural surjective homomorphisms $\psi_n: \pslz \to \mathrm{PSL}_2(\Z/n\Z)$ given by reducing entries modulo $n$. A similar family of groups is given by the upper triangular congruence subgroups $\gamo(n) < \pslz$; these consist of matrices which are congruent to upper triangular matrices modulo $n$. By work of Helling \cite{Helling1,Helling2}, it is known that any maximal arithmetic Fuchsian group which is commensurable with $\pslz$ is obtained by taking the normalizer $N(\gamo(n))$ of $\gamo(n)$ for $n$ square-free, where the normalizer is taken in $\pslr$.

\section{Fuchsian groups}\label{proofs}

In this section, we establish that certain families of Fuchsian groups have the property that each contains infinitely many finite index subgroups with the same trace set as itself. This will be a consequence of the following more general result.

\begin{theorem}\label{general}Let $\Gamma < \pslr$ be a cofinite Fuchsian group with trace set $\tr(\Gamma)$. Let $G_1, \ldots, G_m$ be a finite collection of finitely generated, infinite index subgroups of $\Gamma$ such that 
\[ \bigcup_{i=1}^m \tr(G_i) = \tr(\Gamma), \] 
and for each $1 \leq i \leq m$, let $Q_i$ be a finite-sided, connected fundamental domain for $G_i$. Then for each $i$ there exists $\alpha_i \in \Gamma$ having isometric circle $S_{\alpha_i}$ with interior $I_{\alpha_i}$ such that:
\begin{itemize}
\item for each $1 \leq i \leq m$, $\alpha_i \notin G_i$;
\item for each $1 \leq i \leq m$, we have $\overline{I_{\alpha_i}} \subset \overline{Q_i} \subset \overline{\Hyp}$; and
\item for any $j \neq k$, we have $\overline{I_{\alpha_j^{-1}}} \cap \overline{I_{\alpha_k^{-1}}} = \emptyset$.
\end{itemize}
Hence, the subgroup $H$ of $\Gamma$ generated by the conjugates $\alpha_i G_i \alpha_i^{-1}$ is a subgroup of $\Gamma$ of infinite index, with $\tr(H) = \tr(\Gamma)$. \end{theorem}

\begin{proof} The existence of such $\alpha_i$ follows from the assumptions on the $G_i$ as follows. Choose disjoint open intervals $(x_i, y_i) \subset \overline{Q_i} \cap \R$ and isometries $\alpha_i \in \Gamma$ such that for each $i$, $\overline{I_{\alpha_i}} \cap \R, \overline{I_{\alpha_i^{-1}}} \cap \R \subset (x_i, y_i)$. These $\alpha_i$ can for example be constructed by taking a hyperbolic isometry $\gamma_i$ whose axis endpoints are in the relevant interval (which must exist by the assumption that $\Gamma$ is cofinite), and taking a sufficiently high power for the $\alpha_i$ in order that the isometric spheres satisfy the required condition.

The subgroups $H_i := \alpha_i G_i \alpha_i^{-1}$, and their fundamental domains $\alpha_i(Q_i)$ have the properties that  for any $j \neq k$, $\alpha_j(Q_j) \cup \alpha_k(Q_k) = \Hyp$, and $\alpha_j(Q_j) \cap \alpha_k(Q_k) \neq \emptyset$. As such, the repeated application of Theorem \ref{combthm}, to $H_1$ and $H_2$, and then to $H_1 \ast H_2$ and $H_3$ etc., gives that the subgroup $H$ generated by the $H_i$ has a fundamental domain 
\[ Q = \bigcap_{i=1}^m \alpha_i(Q_i). \]
Since for each $i$, the complement $\Hyp \setminus \alpha_i(Q_i)$ is contained in $I_{\alpha_i^{-1}}$, and these $I_{\alpha_i^{-1}}$ are mutually disjoint, it follows that $Q$ contains the intersection of the exteriors $E_{\alpha_i^{-1}}$, and thus has infinite area. This implies that $H < \Gamma$ is an infinite index subgroup. Finally, each trace of $\tr(\Gamma)$ also belongs to $\tr(H)$, and so $\tr(H) = \tr(\Gamma)$ as required.\end{proof}

The following Lemma will be helpful in applying Theorem \ref{general} to specific examples.

\begin{lemma}\label{infinitelemma} Let $G$ be a Fuchsian group generated by two elements of the form
\[ G = \left\langle g_1 = \begin{pmatrix}1 & m \\ 0 & 1\end{pmatrix}, g_2 = \begin{pmatrix}a & b \\ c & d\end{pmatrix} \right\rangle, \] 
where $c\neq 0$, and suppose that $\frac{|a+d|}{c} < \frac{|m|}{2}$, and that $E_{g_2} \cap E_{g_2^{-1}}$ is a Ford domain for $\left\langle g_2 \right\rangle$. Then if $|m| > \frac{4}{|c|}$, the group $G$ admits a finite-sided Ford fundamental domain of infinite area. \end{lemma}

\begin{proof}The hypothesis that $\frac{|a+d|}{c} < \frac{|m|}{2}$ implies that the centers of the isometric circles of $g_2$ and $g_2^{-1}$ are at most $\frac{|m|}{2}$ apart, and since $|m| > \frac{4}{|c|}$, their most distant endpoints are at most $\frac{|m|}{2} + \frac{2}{|c|} < \frac{|m|}{2} + \frac{|m|}{2} = |m|$ apart. We set
\[ F = \left\{ z \in \Hyp \mid \left| \mathrm{Re}(z) - \frac{a-d}{2c} \right| \leq \frac{|m|}{2} \right\}, \] 
which is a fundamental region for $\left\langle g_1 \right\rangle$, and note that the hypotheses imply that 
\[ Q = E_{g_2} \cap E_{g_2^{-1}} \cap F \]
is a fundamental domain for $G$, and has infinite area. Hence, we are done.\end{proof}

We now apply Theorem \ref{general} to show that $\pslz$ has infinitely many finite index subgroups with the same trace set as itself.

\begin{cor}\label{modular} The modular group $\pslz$ has a finitely generated, infinite index subgroup $H$ with trace set $\tr(H) = \N_0$. Hence, there exist infinitely many finite index subgroups of $\pslz$ with this trace set. \end{cor}

\begin{proof}We apply Theorem \ref{general} to the subgroups $G_0 = \left\langle \left( \begin{smallmatrix} 0 & -1 \\ 1 & 0 \end{smallmatrix} \right) , \left( \begin{smallmatrix} 1 & 5 \\ 0 & 1 \end{smallmatrix} \right)   \right\rangle$, $G_1 = \left\langle \left( \begin{smallmatrix} 1 & -1 \\ 1 & 0 \end{smallmatrix} \right) , \left( \begin{smallmatrix} 1 & 5 \\ 0 & 1 \end{smallmatrix} \right)  \right\rangle,$ and $G_2 = \left\langle \left( \begin{smallmatrix} 2 & -1 \\ 1 & 0 \end{smallmatrix} \right) , \left( \begin{smallmatrix} 1 & 5 \\ 0 & 1 \end{smallmatrix} \right)  \right\rangle.$
Together these groups give all the required traces, via the inclusions: $\tr(G_0) \supset \{ 0, 5, 10, \ldots \},$ because $ \left( \begin{smallmatrix} 1 & 5 \\ 0 & 1 \end{smallmatrix} \right)^n \left( \begin{smallmatrix} 0 & -1 \\ 1 & 0 \end{smallmatrix} \right) = \left( \begin{smallmatrix} 1 & 5n \\ 0 & 1 \end{smallmatrix} \right)  \left( \begin{smallmatrix} 0 & -1 \\ 1 & 0 \end{smallmatrix} \right) = \left( \begin{smallmatrix} 5n & -1 \\ 1 & 0\end{smallmatrix} \right)$; $\tr(G_1) \supset \{ 1, 4, 6, 9, \ldots \},$ because $ \left( \begin{smallmatrix} 1 & 5 \\ 0 & 1 \end{smallmatrix} \right)^n \left( \begin{smallmatrix} 1 & -1 \\ 1 & 0 \end{smallmatrix} \right) = \left( \begin{smallmatrix} 5n+1 & -1 \\ 1 & 0\end{smallmatrix} \right)$, and when $n<0$ we have the traces $\{-4 \sim 4, -9 \sim 9, \ldots \}$; and $\tr(G_2) \supset \{ 2, 3, 7, 8, \ldots \}$, because $\left( \begin{smallmatrix} 1 & 5 \\ 0 & 1 \end{smallmatrix} \right)^n \left( \begin{smallmatrix} 2 & -1 \\ 1 & 0 \end{smallmatrix} \right) = \left( \begin{smallmatrix} 5n+2 & -1 \\ 1 & 0\end{smallmatrix} \right)$, and when $n<0$ we have the traces $\{-3 \sim 3, -8 \sim 8, \ldots \}$. Furthermore, each subgroup satisfies the hypotheses of Lemma \ref{infinitelemma}, and we may take for each fundamental domain $Q_i$ the Ford domain bounded by the isometric spheres of the first generators and the vertical geodesics from $-1$ and $4$ to $\infty$ respectively. Note that for each $0 \leq i \leq 2$, the closure $\overline{Q_i} \subset \Hyp \cup S_\infty$ contains the open interval $(3,4)$. For the conjugating elements, we take
\[ \alpha_0 = \begin{pmatrix} 142 & -545 \\ 37 & -142\end{pmatrix}, \ \  \alpha_1 = \begin{pmatrix} 17 & -58 \\ 5 & -17\end{pmatrix}, \ \ \alpha_2 = \begin{pmatrix} 117 & -370 \\ 37 & -117\end{pmatrix}. \]
Thus the subgroup $H$ is generated by the elements $\left( \begin{smallmatrix}26269 & -100820 \\ 6845 & -26271\end{smallmatrix}\right)$ , $\left( \begin{smallmatrix}-82644 & 317189 \\ -21533 & 82644\end{smallmatrix}\right)$ , $\left( \begin{smallmatrix}424 & -1445 \\ 125 & -426\end{smallmatrix}\right)$ , $\left( \begin{smallmatrix}-782 & 2667 \\ -229 & 781\end{smallmatrix}\right)$, $\left( \begin{smallmatrix}21644 & -68445 \\ 6845 & -21646\end{smallmatrix}\right) $, and $\left( \begin{smallmatrix}-20241 & 64009 \\ -6400 & 20239\end{smallmatrix}\right) $. 
We may now invoke the equivalent definition of LERF to see that this finitely generated, infinite index subgroup $H$ must be the intersection of finite index subgroups of $\pslz$. There must be infinitely many of these finite index subgroups, and the trace set of each contains $\tr(H) = \N_0$, so we are done.\end{proof}
 
\noindent {\bf Remark.} The existence of the subgroup $H$ (and hence a descending chain of finite index subgroups with the same trace set) raises a number of questions. Examples will be constructed in the next section, so we defer some discussion until then. This method fails to produce a subgroup whose trace set is $\N_0 \setminus \{ 0, 1 \}$, which would correspond to a manifold with the same trace set as $\pslz$ except for torsion elements. This is because any finite collection of arithmetic progressions, with common differences $d_1, \ldots, d_m$, which misses $0$ must also miss $\pm D$, where $D$ is the least common multiple of the $d_i$, and all multiples thereof. In general, the method at hand only guarantees the presence of those traces which appear in the arithmetic progressions; for this reason, we apply it only to groups $G$ whose trace set is already a union of arithmetic progressions. Constructing a subgroup with a prescribed trace set not of this form is a more delicate problem.

\noindent {\bf Question.} Does there exist a finite index subgroup of $\pslz$ with trace set $\N_0 \setminus \{ 0,1\} = \{ 2, 3, 4, \ldots \}$?

\noindent It is possible to show the existence of infinitely many families of groups with the same trace set. The next two results show that congruence subgroups of certain forms have infinitely many finite index subgroups of the same trace set. In particular, this shows that there are examples of torsion-free groups with this property.

\begin{cor}\label{gammazero}For each $n \in \N$, the congruence subgroup $\gamo(n) < \pslz$ admits a finitely generated, infinite index subgroup $H_n < \gamo(n)$ such that $\tr(H_n) = \tr(\gamo(n))$. \end{cor}

\begin{proof} Consider the maps
\[ \pslz \stackrel{\pi}{\longrightarrow} \mathrm{PSL}_2(\Z/n\Z) \stackrel{\tr}{\longrightarrow} \Z/n\Z, \]
where $\pi: \pslz \to \mathrm{PSL}_2(\Z/n\Z)$ denotes the natural projection where each entry is reduced modulo $n$, and $\tr$ is the trace map. For an element of $\gamo(n)$,
\[ \begin{pmatrix}a & b \\ c & d\end{pmatrix} \stackrel{\pi}{\longmapsto} \begin{pmatrix}\bar{a} & \bar{b} \\ 0 & \bar{a}^{-1}\end{pmatrix} \stackrel{\tr}{\longmapsto} \bar{a} + \bar{a}^{-1}, \] 
and thus the set $\{ \bar{a} + \bar{a}^{-1} \mid \bar{a} \in \Z/n\Z \mbox{ has a multiplicative inverse} \}$ contains the image of $\gamo(n)$ under the composition $\tr \circ \pi$. Let $\mathcal{S}_n$ denote the preimage in $\Z$ of this set under the standard projection $\Z \to \Z/n\Z$. We claim that $\tr(\gamo(n)) = \mathcal{S}_n$. To see this, let $a \in \{ 1, \ldots, n-1 \}$ be coprime to $n$, and let $d\in \{ 1, \ldots, n-1 \}$ be such that $ad \equiv 1$ mod $n$. Then $ad = 1 + bn$ for some integer $b$, and so the matrix
\[ \begin{pmatrix}a & b \\ n & d\end{pmatrix} \in \gamo(n) \]
has trace $a+d$. Furthermore, the matrices $ \left( \begin{smallmatrix}1 & 1 \\ 0 & 1\end{smallmatrix} \right)^m \left( \begin{smallmatrix}a & b \\ n & d\end{smallmatrix} \right) = \left( \begin{smallmatrix}a+mn & b+md \\ n & d\end{smallmatrix} \right)$ ensure that all integers of the form $(a+d)+mn$, $m \in \Z$, appear as traces of elements of $\gamo(n)$. 

When $n \geq 5$, we generate the subgroup $H_n$ as follows. Let $\{ a_i \} \subset \{ 1, \ldots , n-1 \}$ be a complete set of residue classes coprime to $n$, and for each $i$, let $d_i \in \{ 1, \ldots , n-1 \}$ be such that $a_id_i \equiv 1$ mod $n$; that is, $a_id_i = 1 +b_i n$ for some $b_i \in \Z$. By Lemma \ref{infinitelemma}, since $n \geq 5$, the subgroups $G_i = \left\langle \left( \begin{smallmatrix} 1 & 1 \\ 0 & 1 \end{smallmatrix} \right) , \left( \begin{smallmatrix} a_i & b_i \\ n & d_i \end{smallmatrix} \right)  \right\rangle$ are of infinite index in $\gamo(n)$ for each $i$. 

For $n = 2, 3, 4$ the above method does not generate infinite index subgroups, so we treat these cases individually. For $n=2$, we take the subgroups
\[ G_1 = \left\langle \begin{pmatrix} 1 & 3 \\ 0 & 1 \end{pmatrix} , \begin{pmatrix} 1 & 0 \\ 2 & 1 \end{pmatrix}   \right\rangle, G_2 = \left\langle \begin{pmatrix} 1 & 3 \\ 0 & 1 \end{pmatrix} , \begin{pmatrix} 1 & -1 \\ 2 & -1 \end{pmatrix}   \right\rangle. \]
For $n=3$, we take the subgroups $G_1 = \left\langle \left( \begin{smallmatrix} 1 & 2 \\ 0 & 1 \end{smallmatrix} \right) , \left( \begin{smallmatrix} 1 & 0 \\ 3 & 1 \end{smallmatrix} \right)   \right\rangle, G_2 = \left\langle \left( \begin{smallmatrix} 1 & 2 \\ 0 & 1 \end{smallmatrix} \right) , \left( \begin{smallmatrix} 2 & -1 \\ 3 & -1 \end{smallmatrix} \right)  \right\rangle$. 
For $n=4$, the subgroup $G_1 = \left\langle \left( \begin{smallmatrix} 1 & 2 \\ 0 & 1 \end{smallmatrix} \right) , \left( \begin{smallmatrix} 1 & 0 \\ 4 & 1 \end{smallmatrix} \right)  \right\rangle$ generates all the required traces, and has infinite index in $\gamo(4)$. This treats all cases, and we are done.\end{proof}

\begin{cor}For $p$ prime, the maximal arithmetic Fuchsian group $N(\gamo(p))$ has infinitely many finite index subgroups with the same trace set at itself.\end{cor}

\begin{proof} It is known (see Helling \cite{Helling1, Helling2} and Maclachlan \cite{Maclachlan}) that elements of $N(\gamo(p))$ either belong to $\gamo(p)$ or have the form
\[ \begin{pmatrix}a\sqrt{p} & \frac{b}{\sqrt{p}} \\ c\sqrt{p} & d\sqrt{p}\end{pmatrix},\]
where $a,b,c,d \in \Z$ and the determinant is $1$. To obtain all rational integer traces, we use the same collection of infinite index subgroups which were used in Corollary \ref{gammazero}; when $p \geq 5$, we add to this collection the subgroup
\[ \left\langle \begin{pmatrix}0 & \frac{-1}{\sqrt{p}} \\ \sqrt{p} & 0\end{pmatrix}, \begin{pmatrix}1 & 1 \\ 0 & 1\end{pmatrix} \right\rangle, \]
which has a Ford domain of infinite area, and generates all traces of the form $m\sqrt{p}$ for $m \in \Z$. When $p=2$, we add the two subgroups $ \left\langle \left( \begin{smallmatrix}0 & \frac{-1}{\sqrt{2}} \\ \sqrt{2} & 0\end{smallmatrix} \right), \left( \begin{smallmatrix}1 & 3 \\ 0 & 1\end{smallmatrix} \right) \right\rangle, \left\langle \left( \begin{smallmatrix}\sqrt{2} & \frac{-1}{\sqrt{2}} \\ \sqrt{2} & 0\end{smallmatrix} \right) , \left( \begin{smallmatrix}1 & 3 \\ 0 & 1\end{smallmatrix} \right) \right\rangle$,
and when $p=3$, we add the subgroups $\left\langle \left( \begin{smallmatrix}0 & \frac{-1}{\sqrt{3}} \\ \sqrt{3} & 0\end{smallmatrix} \right), \left( \begin{smallmatrix}1 & 3 \\ 0 & 1\end{smallmatrix} \right) \right\rangle, \left\langle \left( \begin{smallmatrix}\sqrt{3} & \frac{-1}{\sqrt{3}} \\ \sqrt{3} & 0\end{smallmatrix} \right) , \left( \begin{smallmatrix}1 & 3 \\ 0 & 1\end{smallmatrix} \right) \right\rangle$.\end{proof}
 
We remark that the next result is closely related to a theorem of Schmutz \cite{Schmutz}. In particular, Theorem 3 of \cite{Schmutz} gives infinitely many non-isometric surfaces with the same trace set as certain principal congruence subgroups. Our result shows that this holds for every principal congruence subgroup.

\begin{cor}For each $2 \leq n \in \N$, the principal congruence subgroup $\Gamma(n) < \pslz$ admits a finitely generated, infinite index subgroup $H'_n < \Gamma(n)$ such that $\tr(H'_n) = \tr(\Gamma(n))$. \end{cor}

\begin{proof} An element of $\Gamma(n)$ can be given the form
\[ \begin{pmatrix} 1+an & bn \\ cn & 1+dn\end{pmatrix} \]
for integers $a,b,c,d$. Thus traces of such elements have the form $2+(a+d)n$. The determinant is
\[ 1 + (a+d)n +adn^2 - bcn^2 = 1, \]
from which we deduce that $(a+d)n$ is an integer multiple of $n^2$; this implies that $(a+d)$ is a multiple of $n$. Hence, the trace set $\tr(\Gamma(n))$ contains only elements of the form $2 \pm An^2$, for $A \in \Z$. Moreover, for each integer $A$, the matrix
\[ \begin{pmatrix} 1 & An \\ 0 & 1\end{pmatrix} \begin{pmatrix} 1 & 0 \\ n & 1\end{pmatrix}  = \begin{pmatrix} 1+An^2 & An \\ n & 1\end{pmatrix} \]
realizes this trace. Thus $\tr(\Gamma(n)) = \{ An^2 \pm 2 \mid A \in \N_0 \}$. These traces can all be obtained from the subgroup $\left\langle \left( \begin{smallmatrix} 1 & n \\ 0 & 1\end{smallmatrix} \right), \left( \begin{smallmatrix} 1 & 0 \\ n & 1\end{smallmatrix} \right) \right\rangle$, and by Lemma \ref{infinitelemma}, we find that when $n>2$, this subgroup is of infinite index. We take $H'_n$ to be this subgroup. When $n=2$, we take $H'_2 = \left\langle \left( \begin{smallmatrix} 1 & 4 \\ 0 & 1\end{smallmatrix} \right), \left( \begin{smallmatrix} 1 & 0 \\ 2 & 1\end{smallmatrix} \right) \right\rangle$; this subgroup has the required trace set and is of infinite index. \end{proof}

These results also raise interesting questions. One defines a group commensurable with $\pslz$ to be a \emph{congruence subgroup} if it contains a principal congruence subgroup $\Gamma(n)$ for some $n$. In the light of the above results, it is natural to ask:

\noindent {\bf Question.} Does every congruence subgroup commensurable with $\pslz$ admit an infinite index subgroup with the same trace set as itself?

\noindent The above results also rely on the congruences, and related trace information, which are induced by being a congruence subgroup. Since there exist Fuchsian groups commensurable with $\pslz$ which are not congruence subgroups, it is also pertinent to ask whether there exist non-congruence subgroups with this property.

\noindent {\bf Question.} What can we say about non-congruence subgroups commensurable with $\pslz$?

\section{Examples}\label{examples}

In this section, we construct explicit examples of finite index subgroups of $\pslz$ with trace set $\N_0$, as in Corollary \ref{modular}. We will refer to such as \emph{full trace subgroups}. We show that for each $n \in \N$ such that $n \geq 31$, there is a full trace subgroup $G_n$ of $\pslz$ with index $\left[ \pslz : G_n \right]=n$. By considering prime values of $n$, it follows from this that there exist infinitely many maximal proper subgroups of $\pslz$ with full trace. In particular, this shows that the infinitely many subgroups guaranteed by Corollary \ref{modular} cannot be grouped as a single tower of nested subgroups. We also give an alternative construction of full trace finite index subgroups of $\pslz$, the flexibility of which allows for the construction of examples with arbitrarily high multiplicities for finitely many traces.

\noindent {\bf Definition.} Throughout this section we will make use of certain involutions which belong to $\pslz$. For $k \in \Z$, we define by $t_k$ the element

\[ t_k = \begin{pmatrix} 2k+1 & -2k^2-2k-1 \\ 2 & -(2k+1)\end{pmatrix}. \]

Each element $t_k$ has the properties that it fixes the point $\frac{2k+1}{2}+\frac{1}{2}i$, exchanges the points $k$ and $k+1$ on the boundary of $\Hyp$, and has isometric circle the geodesic with endpoints at $k$ and $k+1$. Note that this is an edge of the Farey tessellation.

The following Lemma will allow us to increment the index of certain subgroups by exploiting the Farey tessellation.

\begin{lemma}\label{fareylemma}Suppose that $\Gamma < \pslz$ is a finite index subgroup, of index $n$, with convex fundamental domain $P$ such that one side $s$ of $P$ is an edge of the Farey tessellation, and the two ideal vertices of $s$ are vertices of $P$. Suppose further that the involution $\alpha_s \in \pslz$, with fixed point on $s$, belongs to $\Gamma$. Let $T$ denote the Farey triangle adjacent to $s$ and exterior to $P$. Then there exist
\begin{itemize}\item a subgroup $\Gamma' < \pslz$ of index $n+1$, obtained by replacing $\alpha_s$ with the order 3 rotation of $\pslz$ with fixed point inside $T$; and
\item a subgroup $\Gamma'' < \pslz$ of index $n+3$, obtained by replacing $\alpha_s$ with the two involutions with fixed points in the two other sides of $T$.
\end{itemize} \end{lemma}

\begin{proof}We recall the facts that $\pslz$ acts transitively on the triangles, and the edges, of the Farey tessellation, and that corresponding to each edge $s$ there is a unique involution $\alpha_s$ whose fixed point lies on $e$. Therefore, by conjugation if necessary, we may assume that the Farey edge $s$ is the edge between $0$ and $1$, and that $P$ lies exterior to $s$ (see Figure \ref{farey01}). The involution with fixed point on $s$ is then 
\[ \alpha_s = \begin{pmatrix}1 & -1 \\ 2 & -1\end{pmatrix}, \]
\begin{figure}[htb]
\begin{center}
\includegraphics[scale=0.6]{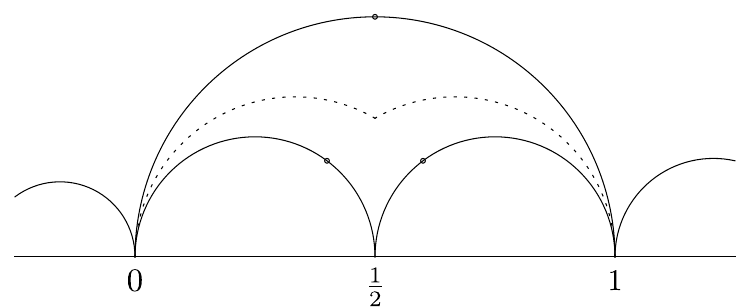}
\caption{We may replace $\alpha_s$ with one order 3 rotation or two involutions}
\label{farey01}
\end{center}
\end{figure}
and the vertices of $T$ lie at $0, \frac{1}{2},$ and $1$. The presentation for $\Gamma$ given by the Poincar\'{e} Polyhedron Theorem then has $\alpha_s$ as a generating side-pairing. The relators in which $\alpha_s$ appears are the relator $\alpha_s^2$, corresponding to the vertex cycle of $\frac{1}{2} + \frac{i}{2}$, and the relator corresponding to the vertex cycle containing $0$ and $1$, in which $\alpha_s$ appears exactly once. As such, we may replace $\alpha_s$ with the order $3$ element
\[ \tau_s = \begin{pmatrix} 2 & -1 \\ 3 & -1\end{pmatrix} \]
which has fixed point $1/2 + \sqrt{-3}/6 \in T$ and which sends $0$ to $1$. The resulting group, $\Gamma'$, has a fundamental domain which consists of the union of $P$ with the third of $T$ exterior to the isometric circles of $\tau_s$. The presentation given by Poincar\'{e} has $\tau_s^3$ in place of $\alpha_s^2$, and $\tau_s$ appears in place of $\alpha_s$ within the other relator where $\alpha_s$ appeared. The new fundamental domain has area $\frac{\pi}{3}$ greater than the area of $P$; hence, the index $\left[ \pslz : \Gamma' \right] = n+1$.

Alternatively, we may also replace $\alpha_s$ with the two involutions
\[ \beta_1 = \begin{pmatrix} 2 & -1 \\ 5 & -2 \end{pmatrix} \ \mbox{and} \ \beta_2 = \begin{pmatrix} 3 & -2 \\ 5 & -3 \end{pmatrix}, \]
whose fixed points lie on the other two sides of $T$. Since $\beta_1(0) = \frac{1}{2}$ and $\beta_2 \left( \frac{1}{2} \right) = 1$, the Poincar\'{e} presentation of the resulting group $\Gamma''$ has the two relators $\beta_1^2$ and $\beta_2^2$ in place of $\alpha_2^2$, and the product $\beta_1 \beta_2$ in place of $\alpha_2$ in the other relator. The union $P \cup T$ is a fundamental domain for $\Gamma''$ corresponding to this presentation. The area of this domain is $\pi$ larger than the area of $P$; hence, the index $\left[ \pslz : \Gamma' \right] = n+3$.\end{proof}

\noindent {\bf Remark.} We believe that Lemma \ref{fareylemma} should be an example of a more general phenomenon that occurs within residually finite groups. The subgroups $G_n$ all share a common infinite index subgroup $H$; the process described in the proof repeatedly ``removes" generators to obtain $H$, and then ``replaces" them with others in order that the index be augmented as required. As such, Theorem \ref{fareythm} below merely exploits the fact that for a certain infinite index subgroup $H$ of $\pslz$, the set of finite index subgroups which contain $H$ includes subgroups of all indices at least 31. It may be of interest to ask how generally this property of $\pslz$, and of $H$, occurs.

\noindent {\bf Remark.} We will refer to the procedure described in Lemma \ref{fareylemma} as \emph{Farey replacement} of Type 1 and Type 2 respectively. Notice that Farey replacement of Type 2 can be repeatedly performed, as the two new sides of the fundamental domain are again edges of the Farey tessellation. 

\begin{theorem}\label{fareythm}For each $n \geq 31$, there is a full trace subgroup $G_n < \pslz$ of index $n$. \end{theorem}

\begin{proof}The subgroup $G_{31}$ is constructed as follows. Begin with the infinite index subgroup generated by the elements $\left( \begin{smallmatrix} 1 & 5 \\ 0 & 1\end{smallmatrix} \right)$ and $\left( \begin{smallmatrix} 3 & -4 \\ 1 & -1\end{smallmatrix} \right)$, with fundamental domain shown in Figure \ref{farey11}. Conjugate this by $t_4$.
\begin{figure}[htb]
\begin{center}
\includegraphics[scale=0.8]{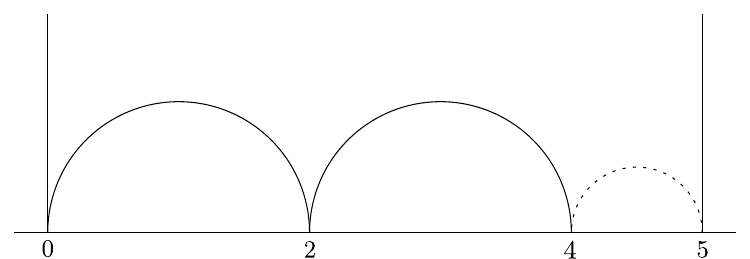}
\caption{Ford domain for first infinite index subgroup}
\label{farey11}
\end{center}
\end{figure}
Next, take the subgroup generated by $\left( \begin{smallmatrix} 1 & 5 \\ 0 & 1\end{smallmatrix} \right)$, $\left( \begin{smallmatrix} 2 & -3 \\ 1 & -1\end{smallmatrix} \right)$, and $t_4$, with fundamental domain shown in Figure \ref{farey12}. Conjugate this by $t_3$. 
\begin{figure}[htb]
\begin{center}
\includegraphics[scale=0.8]{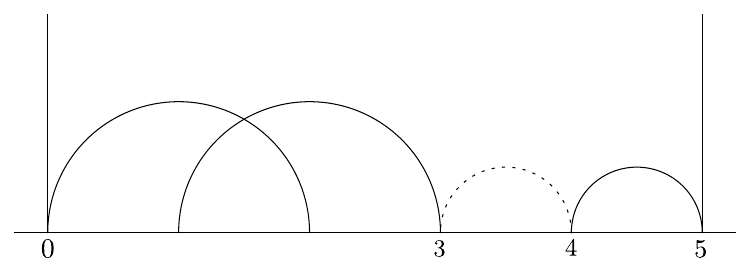}
\caption{The edge from $4$ to $5$ is a Farey edge}
\label{farey12}
\end{center}
\end{figure}
Finally, include the generators $\left( \begin{smallmatrix} 1 & 5 \\ 0 & 1\end{smallmatrix} \right)$, $\left( \begin{smallmatrix} 1 & -2 \\ 1 & -1\end{smallmatrix} \right)$, and $t_2$. The final fundamental domain $P$ for $H_{31}$, which satisfies the Poincar\'{e} Polyhedron Theorem, is shown in Figure \ref{farey13}. This has area $\frac{31}{3}\pi$, and hence index $31$.
\begin{figure}[htb]
\begin{center}
\includegraphics[scale=0.8]{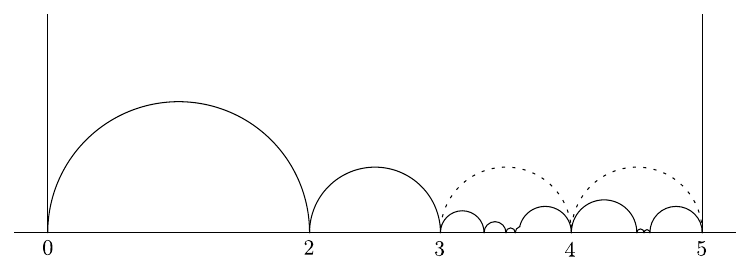}
\caption{Fundamental domain for $G_{31}$; the edge from $2$ to $3$ is a Farey edge}
\label{farey13}
\end{center}
\end{figure}
We observe that the elements $t_2$ and $t_4$ have isometric circles (and sides of $P$) which are edges of the Farey tessellation. By Lemma \ref{fareylemma}, we may perform a Farey replacement of Type 1 on $t_2$ in order to obtain an index $32$ subgroup $G_{32}$. We may perform another Farey replacement of Type 1 on $t_4$ (or, more precisely, on its conjugate which is a generator of $G_{31}$ and $G_{32}$) to obtain an index 33 subgroup $G_{33}$. Starting again with $G_{31}$, we may perform a Farey replacement of Type 2 on $t_2$ to obtain an index $34$ subgroup $G_{34}$. This last replacement produced two new sides on which Farey replacement may be performed; performing one Type 1 replacement gives $G_{35}$, and performing two gives $G_{36}$. 

In general, to obtain a full trace subgroup $G_n$ of index $n > 34$, proceed as follows: if $n \equiv 31$ mod $3$, then perform the required number $\frac{n-31}{3}$ of Type 2 Farey replacements. If $n \equiv 32$ mod $3$, then perform the required number $\frac{n-32}{3}$ of Type 2 Farey replacements, and then perform a single Type 1 Farey replacement. If $n \equiv 33$ mod $3$, then perform the required number $\frac{n-33}{3}$ of Type 2 Farey replacements, and then perform two Type 1 Farey replacements.  \end{proof}

\noindent {\bf Remark.} Using a construction closely related to that used in Theorem \ref{fareythm}, we are able to find a full trace subgroup $G_{25}$ of $\pslz$ with index $25$. Thus far, this is the minimal index in $\pslz$ of known full trace subgroups. However, this subgroup does not possess the flexibility of the subgroup $G_{31}$ that allows us to increment the index. This subgroup is constructed as follows. Conjugate the group generated by $ \left( \begin{smallmatrix} 1 & 5 \\ 0 & 1\end{smallmatrix} \right)$ and $\left( \begin{smallmatrix} 2 & -3 \\ 1 & -1\end{smallmatrix} \right)$ by $\left( \begin{smallmatrix}4 & -17 \\ 1 & -4\end{smallmatrix} \right)$, the group generated by $\left( \begin{smallmatrix} 1 & 5 \\ 0 & 1\end{smallmatrix} \right)$ and $\left( \begin{smallmatrix} 4 & -5 \\ 1 & -1\end{smallmatrix} \right)$ by $t_3$, and then include the generators $\left( \begin{smallmatrix} 1 & 5 \\ 0 & 1\end{smallmatrix} \right)$ and $\left( \begin{smallmatrix} 1 & -2 \\ 1 & -1\end{smallmatrix} \right)$. Note that the resulting fundamental domain has no sides which are Farey edges where the associated involution belongs to the group. The only involution present could be replaced by two more in the manner of Lemma \ref{fareylemma}, but this would remove some traces from the subgroup.

\noindent {\bf Question.} Is there a proper subgroup of $\pslz$ of full trace with index less than $25$? What is the minimal such index?

With the aid of the Magma computer algebra system \cite{magma}, along with an algorithm to determine representatives for all conjugacy classes of a fixed trace such as that given by Fine \cite{Fine}, it is possible to check that there is no full trace subgroup of index less than $14$. As such, the minimal index of such a subgroup must be at least $14$ and at most $25$. This search found an index $13$ subgroup, generated by $\left( \begin{smallmatrix}1 & 8 \\ 0 & 1\end{smallmatrix} \right), \left( \begin{smallmatrix}0 & -1 \\ 1 & 0\end{smallmatrix} \right) , \left( \begin{smallmatrix}1 & -2 \\ 1 & -1\end{smallmatrix} \right) , \left( \begin{smallmatrix}-3 & -10 \\ 1 & 3\end{smallmatrix} \right)$, and $\left( \begin{smallmatrix}13 & 29 \\ 4 & 9\end{smallmatrix}\right)$, which has a trace set containing each integer, with the sole exception of $45$, between $0$ and $100$. We expect that an enumeration of subgroups of index from 14 up to 24 would produce a number of examples similar to this, where the trace set verifiably includes each integer up to some $N$, but one cannot easily find infinite index subgroups of the form of Theorem \ref{general}. As $N$ increases, it becomes increasingly difficult to determine whether or not these subgroups are full trace subgroups.

We now give an alternative construction of full trace subgroups of $\pslz$.

\noindent {\bf Definition.} Given an interval $[a,b] \subset \R$ with $a,b \in \Z$, the operation of \emph{filling} the interval $[a,b]$ will mean including in a subgroup the involutions $t_a, t_{a+1}, \ldots , t_{b-1}$. The isometric circles of these involutions cover the interval $[a,b]$. In case $a=b$, do nothing.

\noindent {\bf Examples.} For each integer $n>5$, we construct a distinct subgroup $H_n$, primarily based around infinite index subgroups containing the parabolic element $\left( \begin{smallmatrix}1 & n \\ 0 & 1\end{smallmatrix} \right)$. The number of infinite index subgroups needed depends on $n$, since each subgroup is designed to generate an arithmetic progression of traces with common difference $n$. 

Begin by setting $j_0 = n-1$, and $i_0 = \left\lfloor \dfrac{n}{2} \right\rfloor$. Define $H^n_{i_0}$ as being generated by the elements
\[ \begin{pmatrix}1 & n \\ 0 & 1\end{pmatrix}, \begin{pmatrix}i_0+1 & -i_0-2 \\ 1 & -1\end{pmatrix}, \]
and those involutions required to fill the intervals $[2,i_0]$ and $[i_0+2,j_0]$. Take as a fundamental domain the Ford domain, between $0$ and $n$. Conjugate this subgroup by $t_{j_0}$. The resulting conjugated fundamental domain will have endpoints at $j_0$ and $j_0+1$, and will cover the interval between them. 

\begin{figure}[htb]
\begin{center}
\includegraphics[scale=0.8]{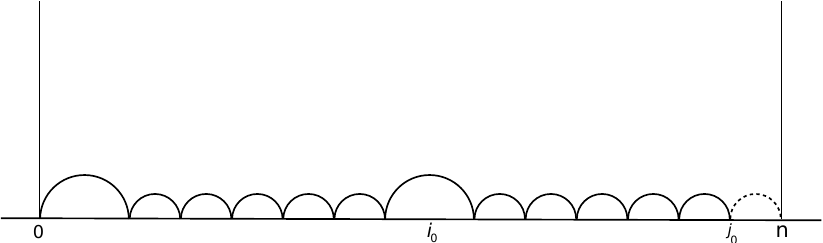}
\caption{First fundamental domain before conjugation by $t_{j_0}$}
\label{farey02}
\end{center}
\end{figure}

Let $i_1 = i_0-1$ and $j_1=j_0-1$. Define $H^n_{i_1}$ as being generated by the elements
\[ \begin{pmatrix}1 & n \\ 0 & 1\end{pmatrix}, \begin{pmatrix}i_1+1 & -i_1-2 \\ 1 & -1\end{pmatrix}, \]
and those involutions required to fill the intervals $[2,i_1]$, $[i_1+2,j_1]$, and $[j_1+1,n]$. Take as a fundamental domain the Ford domain, between $0$ and $n$. Conjugate this subgroup by $t_{j_1}$. The resulting conjugated fundamental domain will have endpoints at $j_1$ and $j_1+1$, and will cover the interval between them.

Proceed iteratively as follows. For each $i_k>1$, define $H^n_{i_k}$ as being generated by the elements 
\[ \begin{pmatrix}1 & n \\ 0 & 1\end{pmatrix}  , \begin{pmatrix}i_k+1 & -i_k-2 \\ 1 & -1\end{pmatrix}, \]
and those involutions required to fill the intervals $[2,i_k]$, $[i_k+2,j_k]$ and $[j_k+1,n]$. Take as a fundamental domain the Ford domain, between $0$ and $n$. Conjugate this subgroup by $t_{j_k}$. Let $i_{k+1}=i_k-1$ and $j_{k+1}= j_k-1$. When $i_k=1$, define $H^n_{1}$ as being generated by the elements $ \left( \begin{smallmatrix}1 & n \\ 0 & 1\end{smallmatrix} \right) , \left( \begin{smallmatrix}2 & -3 \\ 1 & -1\end{smallmatrix} \right)$, and those involutions required to fill the intervals $[3,j_k]$ and $[j_k+1,n]$. Take as a fundamental domain the Ford domain, between $0$ and $n$. Conjugate this subgroup by $t_{j_k}$. Finally, add in the generators $\left( \begin{smallmatrix}1 & n \\ 0 & 1\end{smallmatrix} \right) , \left( \begin{smallmatrix}1 & -2 \\ 1 & -1\end{smallmatrix} \right)$, and those involutions required to fill $[2,j_k-1]$. Do not conjugate these elements. The isometric spheres of these elements cover the remaining space between $0$ and $n$, so we obtain a finite area fundamental domain. 

\noindent {\bf Remark.} The construction above contains a great deal of flexibility, with more as $n$ increases. The involutions introduced by the operation of filling are simply there to ensure that the resulting subgroup has finite index; these involutions may be replaced by other generators of the form
\[ \begin{pmatrix}a & b \\ 1 & c \end{pmatrix} \]
as long as the intervals $[a-1,a+1]$ and $[-c-1,-c+1]$ are contained in the intervals which are filled in the appropriate choice of $H_i^n$. Therefore, to produce an example with at least $M$ non-conjugate elements of trace $x$, take a sufficiently large $N$ (for example, $N >> 2M + 2x$) and introduce $M$ elements of trace $x$ into the subgroup $H_N$.

\section{Bianchi Groups}\label{bianchisection}

In this section, we show that there are results analogous to Corollary \ref{pslzcor} for each Bianchi group. As above, there is a more general result involving any trace set that can be written as the union of finitely many (translates of) sublattices of the ring of integers. For brevity, we will only prove an analogue of Corollary \ref{pslzcor} for each Bianchi group, rather than an analogue of the more general Theorem \ref{mainthm}.

Consider the Bianchi group $\mathrm{PSL}_2(\mathcal{O}_d)$, where $d>0$ is a square-free integer, $\mathcal{O}_d$ denotes the ring of integers in the imaginary quadratic field $\mathbb{Q}(\sqrt{-d})$. It is a standard fact that $\mathcal{O}_d$ is an integer lattice in $\C$ generated by $1$ and $\omega = \sqrt{-d}$ (if $d \equiv 1, 2$ mod $4$) or by $1$ and $\omega = \frac{1+\sqrt{-d}}{2}$ (if $d \equiv 3$ mod $4$).

\begin{bianchi}Given any Bianchi group $\pslod$, there are infinitely many finite index subgroups $\Gamma < \pslod$ with the same complex trace set as $\pslod$.\end{bianchi}

\begin{proof}As in the Fuchsian case, we construct finitely many subgroups which together generate every trace, and then conjugate them so that together they generate a subgroup of infinite index. We then appeal to the fact that the Bianchi groups are LERF. The following general method works whenever there are no non-rational integers $s \in \mathcal{O}_d \setminus \mathbb{Z}$ with modulus $|s|<2$; this includes all Bianchi groups $\pslod$ for $d \neq 1, 2, 3$, cases which we treat afterwards.

The subgroups which we take are the five subgroups of the form
\[ P_x = \left\langle \begin{pmatrix}x & -1 \\ 1 & 0\end{pmatrix}, \begin{pmatrix}1 & 3 \\ 0 & 1\end{pmatrix}, \begin{pmatrix}1 & 3\omega \\ 0 & 1\end{pmatrix} \right\rangle, \]
for $x \in \{ 0, 1, \omega, 1+\omega, 2+\omega \}$. In each case, a Ford domain for $P_i$ is bounded by isometric spheres of radius $1$ centered at $0$ and $x$ respectively, together with vertical planes which form the boundary of a fundamental domain for the parabolics fixing $\infty$. We conjugate the $P_i$ so that we may apply Theorem \ref{combthm}; to do this, we conjugate by involutions $\delta_i$ which have isometric spheres disjoint from those of the $P_i$ and from each other. If $d \equiv 1,2$ mod $4$, we take
\[ \delta_1 = \begin{pmatrix} 38 + 85 \omega & -17 -76\omega -85\omega^2 \\ 85 & -38 -85\omega \end{pmatrix} = \begin{pmatrix} 38 + 85 \omega & 85d -17 -76\omega \\ 85 & -38 -85\omega \end{pmatrix},\]
\[ \delta_\omega = \begin{pmatrix} 7 & -10 \\ 5 & -7 \end{pmatrix}, \delta_{1+\omega} = \begin{pmatrix} 68 & -125 \\ 37 & -68 \end{pmatrix}, \delta_{2+\omega} = \begin{pmatrix} 43 & -50 \\ 37 & -43 \end{pmatrix};\]
and if $d  \equiv 3$ mod $4$, we take
\[ \delta_1 = \begin{pmatrix} -2+3\omega+4d\omega & 1 +4\omega -7\omega^2 -4\omega^2d \\ 4d-1 & 2 -3\omega -4d\omega \end{pmatrix},\]
\[ \delta_\omega = \begin{pmatrix} 7 & -10 \\ 5 & -7 \end{pmatrix}, \delta_{1+\omega} = \begin{pmatrix} 7-5\omega & 5d-10+14\omega \\ 5 & 5\omega-7 \end{pmatrix}, \delta_{2+\omega} = \begin{pmatrix} 43 & -50 \\ 37 & -43 \end{pmatrix}.\]

For the three remaining cases, we choose conjugations specific to each case, and take care because isometric spheres of higher powers of the generators of the subgroups may appear. We treat the cases $d=1,2$ together. We take the same $P_i$ as above, and conjugations
\[ \delta_1 = \begin{pmatrix} 38 + 85 \omega & 85d -17 -76\omega \\ 85 & -38 -85\omega \end{pmatrix}, \delta_\omega = \begin{pmatrix} 43 & -50 \\ 37 & -43 \end{pmatrix}, \]
\[ \delta_{1+\omega} = \begin{pmatrix} 68 & -125 \\ 37 & -68 \end{pmatrix}, \delta_{2+\omega} = \begin{pmatrix} 91 & -101 \\ 82 & -91 \end{pmatrix}. \]
The last case is where $d=3$. We take
\[ \delta_1 = \begin{pmatrix} -2+3\omega+4d\omega & 1 +4\omega -7\omega^2 -4\omega^2d \\ 4d-1 & 2 -3\omega -4d\omega \end{pmatrix},  \delta_\omega = \begin{pmatrix} 43 & -50 \\ 37 & -43 \end{pmatrix}, \]
\[ \delta_{1+\omega} = \begin{pmatrix} 68-37\omega & 99\omega-88 \\ 37 & 37\omega-68 \end{pmatrix}, \delta_{2+\omega} = \begin{pmatrix} 68 & -125 \\ 37 & -68 \end{pmatrix}. \]

By construction, the subgroups $P_0$ and $\delta_x P_x \delta_x$ for $x \in \{ 1, \omega, 1+\omega, 2+\omega \}$ satisfy the hypotheses of Theorem \ref{combthm}, and so the group $H$ generated by $P_0$ and the $\delta_x P_x \delta_x$ is the free product of the generating subgroups, has infinite index in $\pslod$, and has all the same traces as $\pslod$. Since the Bianchi group is LERF, this implies that there exists a proper finite index subgroup $K_d < \pslod$ which contains this subgroup, but not
\[ g = \begin{pmatrix}1 & 1 \\ 0 & 1 \end{pmatrix}, \]
and therefore is a proper, finite index subgroup with the same trace set as the Bianchi group. Moreover, by the alternative formulation of LERF, $H$ can be written as the intersection of infinitely many finite index subgroups of $\pslod$, each of which therefore has trace set the same as $\pslod$. \end{proof}

\bibliographystyle{amsplain}
\bibliography{tracesetrefs}

\providecommand{\bysame}{\leavevmode\hbox to3em{\hrulefill}\thinspace}
\providecommand{\MR}{\relax\ifhmode\unskip\space\fi MR }
\providecommand{\MRhref}[2]{%
  \href{http://www.ams.org/mathscinet-getitem?mr=#1}{#2}
}
\providecommand{\href}[2]{#2}
\begin{thebibliography}{10}

\bibitem{Agol}
I.~Agol, \emph{Tameness of hyperbolic 3--manifolds}, arxiv:math.0405568.

\bibitem{ALR}
I.~Agol, D.~D. Long, and A.~W. Reid, \emph{The {B}ianchi groups are separable
  on geometrically finite subgroups}, Ann. of Math. (2) \textbf{153} (2001),
  no.~3, 599--621. \MR{1836283 (2002e:20099)}

\bibitem{Beardon}
A.~F. Beardon, \emph{The geometry of discrete groups}, Graduate Texts in
  Mathematics, vol.~91, Springer-Verlag, New York, 1983. \MR{698777
  (85d:22026)}

\bibitem{magma}
W.~Bosma, J.~Cannon, and C.~Playoust, \emph{The {M}agma algebra system. {I}.
  {T}he user language}, J. Symbolic Comput. \textbf{24} (1997), no.~3-4,
  235--265, Computational algebra and number theory (London, 1993).
  \MR{MR1484478}

\bibitem{CG}
D.~Calegari and D.~Gabai, \emph{Shrinkwrapping and the taming of hyperbolic
  3-manifolds}, J. Amer. Math. Soc. \textbf{19} (2006), no.~2, 385--446.
  \MR{2188131 (2006g:57030)}

\bibitem{Canary}
R.~D. Canary, \emph{A covering theorem for hyperbolic {$3$}-manifolds and its
  applications}, Topology \textbf{35} (1996), no.~3, 751--778. \MR{1396777
  (97e:57012)}

\bibitem{Fine}
B.~Fine, \emph{Trace classes and quadratic forms in the modular group}, Canad.
  Math. Bull. \textbf{37} (1994), no.~2, 202--212. \MR{1275705}

\bibitem{GenLeu}
S.~Geninska and E.~Leuzinger, \emph{A geometric characterization of arithmetic
  {F}uchsian groups}, Duke Math. J. \textbf{142} (2008), no.~1, 111--125.
  \MR{2397884 (2009a:22007)}

\bibitem{Hall}
M.~Hall, Jr., \emph{Subgroups of finite index in free groups}, Canadian J.
  Math. \textbf{1} (1949), 187--190. \MR{0028836 (10,506a)}

\bibitem{Helling1}
H.~Helling, \emph{Bestimmung der {K}ommensurabilit\"atsklasse der
  {H}ilbertschen {M}odulgruppe}, Math. Z. \textbf{92} (1966), 269--280.
  \MR{0228437}

\bibitem{Helling2}
\bysame, \emph{On the commensurability class of the rational modular group}, J.
  London Math. Soc. (2) \textbf{2} (1970), 67--72. \MR{0277620}

\bibitem{LMNR}
C.~J. Leininger, D.~B. McReynolds, W.~D. Neumann, and A.~W. Reid, \emph{Length
  and eigenvalue equivalence}, Int. Math. Res. Not. IMRN (2007), no.~24, Art.
  ID rnm135, 24. \MR{2377017 (2009a:58046)}

\bibitem{Maclachlan}
C.~Maclachlan, \emph{Groups of units of zero ternary quadratic forms}, Proc.
  Roy. Soc. Edinburgh Sect. A \textbf{88} (1981), no.~1-2, 141--157. \MR{611307
  (83a:10030)}

\bibitem{Maskit}
B.~Maskit, \emph{Kleinian groups}, Grundlehren der Mathematischen
  Wissenschaften [Fundamental Principles of Mathematical Sciences], vol. 287,
  Springer-Verlag, Berlin, 1988. \MR{959135 (90a:30132)}

\bibitem{Pesce}
H.~Pesce, \emph{Compacit\'e de l'ensemble des r\'eseaux isospectraux et
  cons\'equences}, Topology \textbf{36} (1997), no.~3, 695--710. \MR{1422430
  (97m:58201)}

\bibitem{Schmutz}
P.~Schmutz, \emph{Arithmetic groups and the length spectrum of {R}iemann
  surfaces}, Duke Math. J. \textbf{84} (1996), no.~1, 199--215. \MR{1394753
  (98a:11068)}

\bibitem{Sunada}
T.~Sunada, \emph{Riemannian coverings and isospectral manifolds}, Ann. of Math.
  (2) \textbf{121} (1985), no.~1, 169--186. \MR{782558 (86h:58141)}

\bibitem{Takeuchi}
K.~Takeuchi, \emph{A characterization of arithmetic {F}uchsian groups}, J.
  Math. Soc. Japan \textbf{27} (1975), no.~4, 600--612. \MR{0398991 (53
  \#2842)}

\end{thebibliography}

\noindent Department of Mathematics \& Computer Science,\\
Eastern Illinois University,\\
600 Lincoln Avenue,\\
Charleston, IL 61920.\\
Email: gslakeland@eiu.edu

\end{document}